\documentclass{amsart}
\usepackage{amsmath, amssymb}
\usepackage{amsfonts,amscd}
\usepackage[left=3cm,top=3cm,right=2.5cm,bottom=3cm,bindingoffset=0.5cm]{geometry}
\newtheorem{thm}{Theorem}

\newtheorem{prop}[thm]{Proposition}
\renewenvironment{proof}{\par\noindent{\bf Proof.}}{$\square$\par\bigskip}

\def\Z{\mathbb Z}\newtheorem{lemma}[thm]{Lemma}

\newtheorem{cor}[thm]{Corollary}
\def\N{\mathbb N}
\def\Q{\mathbb Q}

\def\log{\operatorname{log}}

\def\m{\operatorname{\underline{m}}}

\begin{document}
\title{Density of solutions to quadratic congruences}
\author[Neha Prabhu]{Neha Prabhu}
\address{Neha Prabhu, IISER Pune, Dr Homi Bhabha Road, Pashan, Pune - 411008, Maharashtra, India}
\email{neha.prabhu@students.iiserpune.ac.in}
\thanks{The research of the author is supported by a PhD scholarship from the National Board for Higher Mathematics, India.}

\maketitle

\begin{abstract}
A classical result in number theory is Dirichlet's theorem on the density of primes in an arithmetic progression. We prove a similar result for numbers with exactly $k$ prime factors for $k>1$. Building upon a proof by E.M. Wright in 1954, we compute the natural density of such numbers where each prime satisfies a congruence condition. As an application, we obtain the density of squarefree $n\leq x$ with $k$ prime factors such that a fixed quadratic equation has exactly $2^k$ solutions modulo $n$.
\end{abstract}

\section{Introduction}
\noindent
The theory of solving a quadratic equation modulo $p$ for $p$ prime has been well studied. Investigating whether a given quadratic equation has solutions, how many there are and calculating what the solutions are, has led to beautiful theorems such as the law of quadratic reciprocity. A related question is the following:

Suppose we fix a quadratic equation $f(x)= x^2 +b x +c$, where $b,c \in \Z$ and would like to know how often the equation $f(x)= 0$ has solutions modulo $N$ if we vary $N$ in a certain range. Let us first look at the case where we vary over primes $p$ not exceeding $x$. Dirichlet, in 1837, showed that solutions would exist for approximately half the primes. In 1896, this was made precise by de la Vall\'{e}e-Poussin. Noting that $f(x)$ has exactly two solutions if and only if the discriminant $D = b^2-4c$ is a square $\bmod \, p$, what Dirichlet and de la Vall\'{e}e-Poussin showed was essentially the following:\\
\begin{prop} \label{Prime}
For a fixed non-square integer $D$, as $x \rightarrow \infty$, $$\dfrac{1}{\pi(x)}\#\left\{ p\leq x, p \text{ prime } : \left(\dfrac{D}{p}\right)=1\right\} \sim \dfrac{1}{2}$$
and
$$ \dfrac{1}{\pi(x)}\#\left\{ p\leq x, p \text{ prime } : \left(\dfrac{D}{p}\right)=-1\right\} \sim \dfrac{1}{2},$$
  where $\left(\dfrac{D}{\cdot}\right)$ is the Kronecker-Legendre symbol and $\pi(x)$ denotes the number of primes not exceeding $x$.
\end{prop}
The main ideas that go into the proof of this result are two classical results: Gauss's law of quadratic reciprocity and the natural density version of Dirichlet's theorem on the infinitude of primes in an arithmetic progression. Dirichlet proved the original theorem around 1836. Later, de la Vall\'{e}e-Poussin proved the statement about the natural density. See Chapter 4, Section IV of \cite{Rib}. He proved that for positive integers $a,q$ with $\gcd(a,q)=1$, % the set of primes congruent to $a\, \bmod q$ has natural density $\dfrac{1}{\phi(q)}.$ In other words, %
 the number of primes $p\leq x$ such that $p \equiv a \,\bmod q$ is asymptotic to $\dfrac{1}{\phi(q)}\dfrac{x}{\log x}$ as $x \rightarrow \infty$. Since then, there have been analogues of this theorem in various settings. For example, by applying the Chebotarev density theorem to the case of cyclotomic extensions $\Q(\zeta_n)$ of $\Q$, we obtain Dirichlet's theorem. The analogue in the case of function fields was proved by H. Kornblum and E. Landau in \cite{Kl}. It is natural to ask if we can extend the result to numbers with $k$ prime factors, $k > 1$. In order to do so, we would first need to talk about the analogue of $\pi(x)$ for numbers with $k$ prime factors, which is defined as follows:
$$ \tau_k(x) := \sum\limits_{{n\leq x}\atop{n=p_1p_2\ldots p_k}} 1,$$ where $n=p_1p_2\ldots p_k$ is the prime factorization of $n$, with $p_1\leq p_2\leq \ldots \leq p_k$. If we add an additional condition that the primes dividing $n$ must be distinct, then we are counting the number of squarefree positive integers not exceeding $x$, having exactly $k$ prime factors and this quantity is denoted by $\pi_k(x)$.

In 1900, E. Landau \cite{Landau} proved that 
\begin{equation}\label{Landau}
\pi_k(x) \sim \tau_k(x) \sim \dfrac{x(\log\log x)^{k-1}}{(k-1)! \log x}.
\end{equation}
In 1954, E. M. Wright gave a simpler proof of this in \cite{Wright}, which appears as Theorem 437 in \cite{HW}. There have been several attempts since then, at deriving a precise estimate with error terms. An exposition of this can be found in Section 7.4 of \cite{MV}.
%We fix $N,k \in \N$ and consider a $k$-tuple $$\m = (m_1,m_2,\ldots, m_k)$$ where each $m_i \in (\Z / N\Z)^{\times}, m_i$'s not necessarily distinct.\\
%Let $\tau_{k,\m}(x)$ denote the number of positive integers $n\leq x$ with $k$ prime factors, counted with multiplicity, satisfying $ p_i \equiv m_i \,\mod\, N \text{ for each } i=1,\ldots,k$. If all the primes are distinct, then $n$ is squarefree. Let $\pi_{k,\m}(x)$ denote the number of such squarefree $n \leq x$. Then we prove
%\begin{thm}
 %\label{mainthm}
 %$$\pi_{k,\m}(x)\sim \tau_{k,\m}(x)\sim \dfrac{1}{\phi(N)^k}\dfrac{x(\log\log x)^{k-1}}{(k-1)! \log x} \qquad(k\geq 2).$$
%\end{thm}

With this in mind, it is natural to ask if we can say something analogous to Proposition \ref{Prime} when $n$ varies over squarefree numbers.  In this paper, we prove the following:
\begin{thm} \label{main-2}
Let $D$ be a non-square integer and $k \in \N$. Fix a $k$-tuple $\underline{\varepsilon} = (\varepsilon_1,\ldots,\varepsilon_k)$ where each $\varepsilon_i = \pm 1$ for each $i= 1,\ldots,k$. Then
$$ \dfrac{1}{\pi_k(x)}\# \bigg\{n\leq x, n=p_1p_2\ldots p_k \text{ with } p_1<p_2<\ldots<p_k: \left(\dfrac{D}{p_i}\right)=\varepsilon_i \text{ for each } i  \bigg\} \sim \dfrac{1}{2^k}, $$ 
where $\pi_k(x)$ denotes the number of squarefree numbers less than $x$ with $k$ prime factors.
\end{thm}
The proof involves an analogous version of Dirichlet's theorem, which is the following:\\
Let us fix $N,k \in \N$ and consider a $k$-tuple $$\m_{[k]} = (m_1,m_2,\ldots, m_k)$$ where each $m_i \in (\Z / N\Z)^{\times},$ the multiplicative group of units in $\Z / N\Z$. The $m_i$'s are not necessarily distinct. Consider positive integers $n\leq x$ with $k$ prime factors, counted with multiplicity. Represent such $n$ as $n=p_1p_2\ldots p_k$ with $p_1 \leq p_2 \leq \ldots \leq p_k$. Let $\tau_{k,\m_{[k]}}(x)$ denote the number of positive integers $n\leq x$ with $k$ prime factors satisfying $ p_i \equiv m_i \,\bmod N \text{ for each } i=1,\ldots,k$. If the primes are distinct, then $n$ is squarefree. Let $\pi_{k,\m_{[k]}}(x)$ denote the number of such squarefree $n \leq x$. Then we prove
\begin{thm}
 \label{mainthm}
 $$\pi_{k,\m_{[k]}}(x)\sim \tau_{k,\m_{[k]}}(x)\sim \dfrac{1}{\phi(N)^k}\dfrac{x(\log\log x)^{k-1}}{(k-1)! \log x} \qquad(k\geq 2).$$
\end{thm}
%We first prove Theorem \ref{mainthm} and use it to prove Theorem \ref{main-2}. Along with the proof of Theorem \ref{main-2}, we give the proof of Proposition \ref{Prime} for the sake of completeness.\\
\noindent
\textbf{Remark:}
Note that for $k=1$, Theorem \ref{mainthm} is exactly the statement of de la Vall\'{e}e-Poussin's version of Dirichlet's density theorem. The prime number theorem, the non-vanishing of $L(1,\chi)$ and the orthogonality relations satisfied by Dirichlet characters are the key results that are used in the proof. Similarly, in the proof of Theorem \ref{mainthm}, the natural density theorem of de la Vall\'{e}e-Poussin and Landau's result stated in Equation (\ref{Landau}) play a significant role. In fact, we essentially use the technique used by Wright in \cite{Wright} and an orthogonality relation satisfied by the Dirichlet characters to obtain the result.

The paper is divided as follows. We start by proving Theorem \ref{mainthm}. The second section sets the stage by introducing functions and notation that will be used in the proof. In the next section we prove the non-trivial part of the proof of Theorem \ref{mainthm} in detail. With Section 4, we wrap up the proof of this theorem. After that, the proof of Proposition \ref{Prime} is given for the sake of completeness and finally, the Theorem \ref{main-2} is presented, which uses Proposition \ref{Prime} and Theorem \ref{mainthm}. We also include two corollaries of the theorem.
\section{Preliminaries}
The following notation will be used in the proof of Theorem \ref{mainthm}:
\begin{enumerate}
\item{ We write $\m_{[k]}$ to denote a $k$-tuple $(m_1,m_2,\ldots,m_k)$.
}
\item{ We use $\m^i_{[k-1]}$ to denote the $(k-1)$-tuple formed by removing the $i^{th}$ coordinate of the $k$-tuple $\m_{[k]}$ under consideration.
}
\item{
Henceforth, the sum $\sum\limits_{p_1p_2\ldots p_k \leq x}$ is taken over all sets of primes $\{p_1, p_2\ldots p_k\}$ such that $p_1 p_2\ldots p_k \leq x$, two sets being considered different even if they differ only in the order of primes.
}
\item{For a fixed $\m_{[k]}$, we write \\ $$\sum\limits_{p_1p_2\ldots p_k \leq x}\boldsymbol{\chi}_{\m_{[k]}} := \sum\limits_{p_1p_2\ldots p_k \leq x}\sum\limits_{\sigma \in S'_k}\sum\limits_{\chi} \overline{\chi(m_{\sigma (1)})}\chi(p_1) \sum\limits_{\chi} \overline{\chi(m_{\sigma (2)})}\chi(p_2)\ldots \sum\limits_{\chi} \overline{\chi(m_{\sigma (k)})}\chi(p_k)$$ where
\\1. The set $S'_k$ is the subset of the symmetric group on $k$ symbols consisting of those permutations that give rise to distinct permutations of $\{m_1,m_2,\ldots,m_k\}$.\\
2. The sum $\sum\limits_{\chi}$ runs over the Dirichlet characters modulo $N$.
}
\end{enumerate}
\textbf{Note. }We have the following orthogonality relation satisfied by Dirichlet characters mod $N$:
$$\sum\limits_{\chi} \overline{\chi(m)}\chi(n) = 
\begin{cases}
\phi(N) &\text{if } m \equiv n\bmod N \\
0      & \text{otherwise.} 
\end{cases}
$$ 
 It is easy to see that, for a fixed $n=p_1p_2\ldots p_k$ and $\sigma \in S_k'$, the product $$\sum\limits_{\chi} \overline{\chi(m_{\sigma (1)})}\chi(p_1) \sum\limits_{\chi} \overline{\chi(m_{\sigma (2)})}\chi(p_2)\ldots \sum\limits_{\chi} \overline{\chi(m_{\sigma (k)})}\chi(p_k)$$ is non-zero if and only if  $p_i \equiv m_{\sigma(i)}$ for all $i=1,\ldots,k$. The orthogonality relation tells us that this non-zero quantity is $\phi(N)$ for each $i$. Therefore, for each $n=p_1p_2\ldots p_k$, the inner double sum is $\phi(N)^k$ if, for some $\sigma \in S'_k$, we have  $p_i \equiv m_{\sigma(i)}$ for every $i$ and zero otherwise. Observe that this can happen for at most one permutation $\sigma \in S_k'$.

\bigskip
The following are auxiliary functions that will appear in the proof:\\
$1. \Pi_{k,{\chi},\m_{[k]}} (x) = \dfrac{1}{\phi(N)^k} \sum\limits_{p_1p_2\ldots p_k \leq x}  \boldsymbol{\chi}_{\m_{[k]}}.$
\\
\\
$2. \vartheta_{k,\chi,\m_{[k]}} (x) = \dfrac{1}{\phi(N)^k} \sum\limits_{p_1p_2\ldots p_k \leq x} \log(p_1p_2\ldots p_k)\boldsymbol{\chi}_{\m_{[k]}}.$
\\
\\
$3. L_{k,\chi,\m_{[k]}} (x) = \dfrac{1}{\phi(N)^k}\sum\limits_{p_1p_2\ldots p_k \leq x} \dfrac{1}{(p_1p_2\ldots p_k)}\boldsymbol{\chi}_{\m_{[k]}}.$
\\

%Observe that $ \Pi_{k,\chi,\m} (x)$ counts the number of $n=p_1p_2\ldots p_k$ such that  for \textit{some} permutation $\sigma \in S_k$, we have  $p_i \equiv m_{\sigma(i)}$ for all $i$. \\
By Dirichlet's theorem, we know that for $i\neq j$ the number of primes $p \equiv m_{\sigma(i)} \,\bmod N$ is asymptotically the same as the number of primes $p \equiv m_{\sigma(j)}\, \bmod N$. Thus, if we fix a permutation of $\{m_1, m_2,\ldots ,m_k\}$ , then the number of ordered sets $\{p_1, p_2\ldots p_k\}$ so that $p_i \equiv m_i\, \bmod N$ is equal to $ \dfrac{1}{M} \Pi_{k,\chi,\m_{[k]}} (x)$, where $M$ is the number of distinct permutations of the multiset $\{ m_1, m_2, \ldots m_k\}.$ 

\section{Towards a generalization of Dirichlet's density theorem}
The proof of Theorem \ref{mainthm} comes down to proving the following:
\begin{prop}
 \label{theta}
$\vartheta_{k,\chi,\m_{[k]}} (x) \sim \dfrac{M}{\phi(N)^k}kx(\log\log x)^{k-1} \quad (k\geq 2).
$
\end{prop}
\noindent
The proof of this proposition will follow after a series of lemmas.\\
First, we prove a recursive relation for $\vartheta_{k,\chi,\m_{[k]}} (x)$:
\begin{lemma} \label{recursion-theta}
For $k\geq 1,$ $$ k\vartheta_{k+1,\chi,\m_{[k+1]}} (x) = (k+1)\sum\limits_{p\leq x} \dfrac{1}{\phi(N)}{\mathop{\sum}}'_{i} \left( \sum\limits_{\chi} \overline{\chi(m_i)}\chi(p)\vartheta_{k,\chi,\m^i_{[k]}} \left(\dfrac{x}{p}\right)\right),$$
where the dash on top of the second summation symbol denotes that only those $i=1,\ldots,k$ are counted so that the $\m^i_{[k]}$ are distinct.
\end{lemma}
\begin{proof}
\begin{equation*}
\begin{split}
(k+1) \vartheta_{k+1,\chi,\m_{[k+1]}} (x) &= \dfrac{1}{\phi(N)^{k+1}} \sum\limits_{p_1p_2\ldots p_{k+1} \leq x} (k+1) \log(p_1p_2\ldots p_{k+1})\boldsymbol{\chi}_{\m_{[k+1]}}\\
&= \dfrac{1}{\phi(N)^{k+1}} \sum\limits_{p_1p_2\ldots p_{k+1} \leq x} \boldsymbol{\chi}_{\m_{[k+1]}} (\log p_1 + \log(p_2p_3\ldots p_{k+1}) + \log p_2+ \log(p_1p_3\ldots p_{k+1}) \\
& \qquad \qquad \qquad \qquad \qquad +\ldots + \log p_{k+1}+ \log(p_1p_2\ldots p_{k}) )\\
&= \dfrac{1}{\phi(N)^{k+1}} \sum\limits_{p_1p_2\ldots p_{k+1} \leq x} \log(p_1p_2\ldots p_{k+1})\boldsymbol{\chi}_{\m_{[k+1]}} \\
& \quad \quad \quad + \dfrac{1}{\phi(N)^{k+1}} \sum\limits_{p_1p_2\ldots p_{k+1} \leq x} (\log(p_2p_3\ldots p_{k+1}) + \ldots+\log(p_1p_2\ldots p_{k}) )\boldsymbol{\chi}_{\m_{[k+1]}}\\
&= \dfrac{1}{\phi(N)^{k+1}} \sum\limits_{p_1p_2\ldots p_{k+1} \leq x} \log(p_1p_2\ldots p_{k+1})\boldsymbol{\chi}_{\m_{[k+1]}} + \dfrac{(k+1)}{\phi(N)^{k+1}} \sum\limits_{p_1p_2\ldots p_{k+1} \leq x} \log(p_2p_3\ldots p_{k+1}) \boldsymbol{\chi}_{\m_{[k+1]}}
\end{split}
\end{equation*}

The first sum is just $\vartheta_{k+1,\chi,\m_{[k+1]}} (x)$ and this reduces the left hand side to $k\vartheta_{k+1,\chi,\m_{[k+1]}} (x)$.\\
In the second sum, observe that the $\boldsymbol{\chi}_{\m_{[k+1]}}$ appearing is a $(k+1)$-tuple. Collecting the terms corresponding to $p_1$ in $\boldsymbol{\chi}_{\m_{[k+1]}}$, the second term can be written as follows.\\
$$\sum\limits_{p_1p_2\ldots p_{k+1} \leq x} \log(p_2p_3\ldots p_{k+1}) \boldsymbol{\chi}_{\m_{[k+1]}} = {\mathop{\sum}}'_{i} \sum\limits_{p_1p_2\ldots p_{k+1} \leq x} \log(p_2p_3\ldots p_{k+1}) \boldsymbol{\chi}_{{\m}^i_{[k]}}\left(\sum\limits_{\chi}\overline{\chi(m_i)}\chi(p_1)\right).$$
\text{Simplifying, we get}
$$k\vartheta_{k+1,\chi,\m_{[k+1]}} (x) = (k+1)\sum\limits_{p\leq x} \dfrac{1}{\phi(N)}{\mathop{\sum}}'_{i} \left( \sum\limits_{\chi} \overline{\chi(m_i)}\chi(p)\vartheta_{k,\chi,\m^i_{[k]}} \left(\dfrac{x}{p}\right)\right).$$

\end{proof}
Similarly, we prove a recursion formula for the function $L_{k,\chi,\m_{[k]}}(x)$:
\begin{lemma} \label{recursion-L}
Let $L_{0,\chi,\m_{[0]}} (x)= 1.$ Then for $k\geq 1$, $$L_{k,\chi,\m_{[k]}}(x) = \sum\limits_{p\leq x} \dfrac{1}{p}{\mathop{\sum}}'_{i} \dfrac{1}{\phi(N)}\sum\limits_{\chi} \overline{\chi(m_i)}\chi(p) L_{k-1,\chi,\m^i_{[k-1]}}\left(\dfrac{x}{p}\right),$$
where the dash on top of the second summation symbol is as defined in Lemma \ref{recursion-theta}.
\end{lemma}
This follows directly from the definitions.\\

Let
\begin{equation} \label{f,theta,L}
f_{k,\chi,\m_{[k]}}(x) = \phi(N)^k \vartheta_{k,\chi,\m_{[k]}} (x) -x k\phi(N)^{k-1} {\mathop{\sum}}'_{i} L_{k-1,\chi,\m^i_{[k-1]}} (x). 
\end{equation} 
The idea is to first estimate $f_{k,\chi,\m_{[k]}}(x)$ and $L_{k,\chi,\m_{[k]}}(x)$. Plugging in these estimates into Equation (\ref{f,theta,L}) would then give an asymptotic formula for $\theta_{k,\chi,\m_{[k]}} (x)$ thus proving Proposition \ref{theta}. With this in mind, we first prove a recursion formula for $f_{k,\chi,\m_{[k]}}(x).$
\begin{lemma} \label{recursion-f}
 $$kf_{k+1,\chi,\m_{[k+1]}}(x) = (k+1) \sum\limits_{p\leq x} {\mathop{\sum}}'_{i} \sum\limits_{\chi} \overline{\chi(m_i)}\chi(p) f_{k,\chi,\m^i_{[k]}}\left(\dfrac{x}{p}\right).$$
 \end{lemma}
\begin{proof}\\
From the definition of $f_{k,\chi,\m_{[k]}}(x)$, we have\\
$kf_{k+1,\chi,\m_{[k+1]}}(x) = k \phi(N)^{k+1} \vartheta_{k+1,\chi,\m_{[k+1]}} (x) -x k(k+1) \phi(N)^k {\mathop{\sum}}'_{i} L_{k,\chi,\m^i_{[k]}} (x).$\\
We evaluate the two summands using Lemma \ref{recursion-theta} and Lemma \ref{recursion-L} proved above.\\
By Lemma \ref{recursion-theta} we have $$k \phi(N)^{k+1} \vartheta_{k+1,\chi,\m_{[k+1]}} (x) = \phi(N)^{k+1}(k+1)\sum\limits_{p\leq x} \dfrac{1}{\phi(N)}{\mathop{\sum}}'_{i}\left( \sum\limits_{\chi} \overline{\chi(m_i)}\chi(p)\vartheta_{k,\chi,\m^i_{[k]}} \left(\dfrac{x}{p}\right)\right),$$
which simplifies to
$$(k+1) \sum\limits_{p\leq x} {\mathop{\sum}}'_{i}\sum\limits_{\chi} \overline{\chi(m_i)}\chi(p) \left[ \phi(N)^k \vartheta_{k,\chi,\m^i_{[k]}} \left(\dfrac{x}{p}\right) \right].$$
Also using Lemma \ref{recursion-L},\\
$${\mathop{\sum}}'_{i} L_{k,\chi,\m^i_{[k]}} (x) = \sum\limits_{i=1}^{k+1} \sum\limits_{p\leq x} \dfrac{1}{p} {\mathop{\sum}}'_{j} \dfrac{1}{\phi(N)} \sum\limits_{\chi} \overline{\chi(m_j)}\chi(p) L_{k-1,\chi,\m^{i,j}_{[k-1]}} \left(\dfrac{x}{p}\right),$$
where $\m^{i,j}_{[k-1]}$ denotes $\m^{i}_{[k]}$ with the $j$-th coordinate removed and ${\mathop{\sum}}'_{j}$ denotes that only distinct $\m^{i,j}_{[k-1]}$ are counted. 

Therefore, 
$$x k(k+1) \phi(N)^k {\mathop{\sum}}'_{i} L_{k,\chi,\m^i_{[k]}} (x) = (k+1) \sum\limits_{p\leq x} {\mathop{\sum}}'_{i}\sum\limits_{\chi} \overline{\chi(m_i)}\chi(p)\left[k\phi(N)^{k-1} \dfrac{x}{p} {\mathop{\sum}}'_{j} L_{k-1,\chi,\m^{i,j}_{[k-1]}} \left(\dfrac{x}{p}\right)\right].$$
Putting the two summands together, we obtain the result.
\end{proof}
\noindent
Next, we use Lemma \ref{recursion-f} to get an estimate for $f_{k,\chi, \m_{[k]}} (x)$.
\begin{lemma} \label{f-estimate}
Let $k\geq 1$. Then 
$$f_{k,\chi,\m_{[k]}} (x) = o\{x (\log\log x)^{k-1}\}.$$
\end{lemma}
\begin{proof}
We induct on $k$.\\
When $k=1$, writing $\m_{[1]} =m$,  $$f_{1,\chi,m}(x) = \phi(N)\vartheta_{1,\chi,m}(x) - x.$$
From Dirichlet's theorem on the density of primes in an arithmetic progression, $\vartheta_{1,\chi,m} (x)\sim \frac{1}{\phi(N)}x$ and so  $$f_{1,\chi,m}(x) = o(x).$$
Suppose the claim were true for $k=K$, where $K> 1$. This means for any $\varepsilon > 0$, there exists $x_0 = x_0(K,\varepsilon)$ such that $$|f_{K,\chi,\m_{[K]}}(x)| < \varepsilon x(\log\log x)^{K-1} \quad \qquad \forall x\geq x_0.$$
Also, for $1\leq x <x_0$, from the definition of $f_{K,\chi,\m_{[K]}}$, we can find a real number $D$ depending on $K,\varepsilon$ so that $$|f_{K,\chi,\m_{[K]}}(x)|<D.$$
Using the above we deduce \\
\begin{enumerate}
\item{For $p\leq \frac{x}{x_0}$,
\begin{equation*}
\begin{split}
\sum\limits_{p\leq \frac{x}{x_0}} \left| \sum\limits_{i=1}^{K+1} \sum\limits_{\chi} \overline{\chi(m_i)}\chi(p)f_{K,\chi,\m^i_{[K]}} \left(\dfrac{x}{p}\right) \right| &< (K+1) \phi(N)\varepsilon(\log\log x)^{K-1}\sum\limits_{p\leq \frac{x}{x_0}}\dfrac{x}{p}\\
&<(K+2)\phi(N)\varepsilon x(\log\log x)^K \qquad \text{for $x$ large enough.}
\end{split}
\end{equation*}

}
\item{
For $\frac{x}{x_0}<p\leq x$,
$$\sum\limits_{\frac{x}{x_0}<p\leq x} \left| \sum\limits_{i=1}^{K+1} \sum\limits_{\chi} \overline{\chi(m_i)}\chi(p)f_{K,\chi,\m^i_{[K]}} \left(\dfrac{x}{p}\right)\right| < (K+1)\phi(N)D\pi(x) <(K+1)\phi(N)Dx.$$

}
\end{enumerate}
Hence, using Lemma \ref{recursion-f} and the simple inequality $(K+1) < 2K$ for $K>1$, we have\\
$K|f_{K+1,\chi,\m_{[K+1]}}(x)| < 2K\phi(N)x((K+2)\varepsilon (\log\log x)^k + (K+1)D).$\\
Thus, for $x>x_1(D,\varepsilon,K)$ we conclude 
$$|f_{K+1,\chi,\m_{[K+1]}}(x)| < 2(K+2)\phi(N)\varepsilon x (\log\log x)^K.$$
Since $\varepsilon$ was arbitrary, the claim follows for all $k\in \N$ by induction.
\end{proof}
To complete the proof of Proposition \ref{theta}, it suffices to prove:
\begin{lemma} \label{L-estimate}
$$L_{k,\chi,\m_{[k]}}(x) \sim \dfrac{M}{\phi(N)^k}(\log\log x)^k.$$
\end{lemma}

 \begin{proof}
 Recall that $$L_{k,\chi,\m_{[k]}}(x) = \dfrac{1}{\phi(N)^k}\sum\limits_{p_1p_2\ldots p_k \leq x} \dfrac{1}{(p_1p_2\ldots p_k)}\boldsymbol{\chi}_{\m_{[k]}}$$
 $$= \dfrac{1}{\phi(N)^k}\sum\limits_{p_1p_2\ldots p_k \leq x} \dfrac{1}{(p_1p_2\ldots p_k)} \sum\limits_{\sigma \in S_k}\sum\limits_{\chi} \overline{\chi(m_{\sigma (1)})}\chi(p_1) \sum\limits_{\chi} \overline{\chi(m_{\sigma (2)})}\chi(p_2)\ldots  \sum\limits_{\chi}\overline{\chi(m_{\sigma (k)})}\chi(p_k)$$
 and that $M$ is the number of permutations of the (possible) multiset $\{m_1,m_2,\ldots,m_k\}.$
 \\
 
 We observe that the following hold:\\
 Given a squarefree number $n$ with $k$ factors, if  each prime $p$ dividing $n$ satisfies $p \leq x^{1/k}$ then $n\leq x$. This leads us to write $$L_{k,\chi,\m_{[k]}}(x) \geq M\prod\limits_{i=1}^{k} \sum\limits_{p\leq x^{1/k}}\dfrac{1}{p}\left(\dfrac{1}{\phi(N)}\sum\limits_{\chi} \overline{\chi(m_i)}\chi(p)\right),$$
 i.e.,
 $$L_{k,\chi,\m_{[k]}} \geq M\prod\limits_{i=1}^{k}\sum\limits_{{p\leq x^{1/k}} \atop {p\equiv m_i\, \bmod  N}}\dfrac{1}{p}.$$
 Similarly, if $n=p_1p_2\ldots p_k$ is less than $x$ then each $p_i \leq x,$ which gives us an upper bound:
 $$L_{k,\chi,\m_{[k]}}(x) \leq M\prod\limits_{i=1}^{k} \sum\limits_{p\leq x}\dfrac{1}{p}\left(\dfrac{1}{\phi(N)}\sum\limits_{\chi} \overline{\chi(m_i)}\chi(p)\right)=M\prod\limits_{i=1}^{k}\sum\limits_{{p\leq x} \atop {p\equiv m_i\, \bmod  N}}\dfrac{1}{p}.$$
 It is known (see for example \cite{CPom}) that for any $a$ coprime to $N$, $$\sum\limits_{{p\leq x} \atop {p\equiv a\, \bmod N}} \dfrac{1}{p} \sim \dfrac{1}{\phi(N)}\log\log x.$$
 Thus, $L_{k,\chi,\m_{[k]}}(x)$ is bounded below and above by functions that are each asymptotic to $\dfrac{M}{\phi(N)^k} (\log\log x)^k$, implying that $$L_{k,\chi,\m_{[k]}}(x) \sim \dfrac{M}{\phi(N)^k} (\log\log x)^k.$$
 \end{proof}
 \noindent
 Finally, Proposition \ref{theta} follows by using Lemma \ref{f-estimate} and Lemma \ref{L-estimate} in Equation (\ref{f,theta,L}). 
 \\
 \\
 \textbf{Remark: }Some care needs to be taken while applying Lemma \ref{L-estimate}. The term  $\sum\limits_{i=1}^{k} L_{k-1,\chi,\m^i_{[k-1]}} (x)$ appearing in Equation (\ref{f,theta,L}) involves number of distinct permutations of $\m^i_{[k-1]}$, whereas $M$ appearing in Proposition \ref{theta} is the number of distinct permutations of $\m_{[k]}$. This is resolved by using the following simple fact:\\
 Let $k_1 + k_2 +\ldots +k_m =n.$ Then $$\dfrac{n!}{k_1!k_2!\ldots k_m!} = \dfrac{(n-1)!}{(k_1-1)!k_2!\ldots k_m!} + \dfrac{(n-1)!}{k_1!(k_2 -1)!\ldots k_m!} + \ldots + \dfrac{(n-1)!}{k_1!k_2!\ldots (k_m -1)!}.$$
 \bigskip
 We are now ready to prove the theorem.
 \section{Proof of Theorem \ref{mainthm}}
 \noindent
 By partial summation we have %and the sequence being 
%$$ a_n = 
%\begin{cases}
%\dfrac{1}{\phi(N)^k}\chi_{\m} &\text{  if $n$ is squarefree with $k$ factors}\\
%0 &\text{otherwise}
%\end{cases}$$ 
$$
\vartheta_{k,\chi,\m_{[k]}} (x) = \Pi_{k,\chi,\m_{[k]}} (x) \log x- \int\limits_{2}^{x} \dfrac{\Pi_{k,\chi,\m_{[k]}} (t)}{t} dt.$$
Clearly, $\Pi_{k,\chi,\m_{[k]}} (t) = O(t)$ and therefore,
$$ \int\limits_{2}^{x} \dfrac{\Pi_{k,\chi,\m_{[k]}} (t)}{t} dt = O(x).$$
Hence, for $k\geq 2$, by Proposition \ref{theta},\\
$\Pi_{k,\chi,\m_{[k]}} (x) = \dfrac{\vartheta_{k,\chi,\m_{[k]}} (x)}{\log x} + O\left( \dfrac{x}{\log x}\right) \sim \dfrac{M}{\phi(N)^k}\dfrac{kx(\log\log x)^{k-1}}{\log x}.$ Thus,
\begin{equation} \label{Pi-estimate}
\dfrac{1}{M} \Pi_{k,\chi,\m_{[k]}} (x) \sim \dfrac{1}{\phi(N)^k}\dfrac{kx(\log\log x)^{k-1}}{\log x}.
\end{equation}
We now relate this to the functions $\pi_{k,\m_{[k]}}(x)$ and $\tau_{k,\m_{[k]}}(x)$.
It is easy to see that $$k!\pi_{k,\m_{[k]}}(x) \leq \dfrac{1}{M} \Pi_{k,\chi,\m_{[k]}} (x) \leq k!\tau_{k,\m_{[k]}}(x).$$
We have two cases to consider.\\
\textbf{Case 1:} The units $m_1,m_2,\ldots m_k$ are distinct. \\
Then $\boldsymbol{\chi}_{\m_{[k]}} =0$ unless $p_1,p_2,\ldots,p_k$ are all distinct. This forces the following equality:
$$k!\pi_{k,\m_{[k]}}(x) = \dfrac{1}{M} \Pi_{k,\chi,\m_{[k]}} (x) = k!\tau_{k,\m_{[k]}}(x),$$ so using Equation (\ref{Pi-estimate}) we are done.\\
\bigskip
\textbf{Case 2:} At least two of the $m_i$ are equal.\\
Certainly, in this case we include those $n= p_1\ldots p_k$ so that at least two of the primes are equal. The number of such $n\leq x$ is $\tau_{k,\m_{[k]}}(x) - \pi_{k,\m_{[k]}}(x).$ These $n$ can be expressed in the form $n= p_1\ldots p_k$ with $p_{k-1} = p_k$ and $\m_{[k]}$ with $m_{k-1} = m_k.$ Therefore, we have 
$$ \tau_{k,\m_{[k]}}(x) - \pi_{k,\m_{[k]}}(x) \leq \dfrac{1}{M} \sum\limits_{p_1p_2\ldots p_{k-1}^2\leq x} \dfrac{1}{\phi(N)^k} \boldsymbol{\chi}_{\m_{[k]}} \leq  \dfrac{1}{M} \sum\limits_{p_1p_2\ldots p_{k-1}\leq x} \dfrac{1}{\phi(N)^{k-1}} \boldsymbol{\chi}_{\m^k_{[k-1]}} = \dfrac{1}{M}\Pi_{k-1,\chi,\m_{[k-1]}} (x).$$
Since $\dfrac{1}{M}\Pi_{k-1,\chi,\m_{[k-1]}} (x)$ is $o\left(\dfrac{1}{M}\Pi_{k,\chi,\m_{[k]}} (x)\right)$, from our observation above, we have
$$\pi_{k,\m_{[k]}}(x)\sim \tau_{k,\m_{[k]}}(x)\sim \dfrac{1}{\phi(N)^k}\dfrac{x(\log\log x)^{k-1}}{(k-1)! \log x} \qquad(k\geq 2)$$ thus proving the theorem in this case as well.
\bigskip
 \section{Proofs of Proposition \ref{Prime} and Theorem \ref{main-2}}
 \noindent
 In order to prove Proposition \ref{Prime}, we note that it suffices to prove the result for $p$ odd, since $2$ is the only even prime and the density of finite sets is zero. Thus we will assume that $p$ is odd in the proof.  
\\
\textbf{Proof of Proposition \ref{Prime}.}\\
 Let $D =\pm q^{a_1}_1 q^{a_2}_2\ldots q^{a_m}_m$ be the decomposition of $D$. Then, by the multiplicative property of the Legendre symbol, we have $$\left(\dfrac{D}{p}\right)= \left(\dfrac{\pm 1}{p}\right) \left(\dfrac{q_1}{p}\right)^{a_1}\left(\dfrac{q_2}{p}\right)^{a_2}\ldots \left(\dfrac{q_m}{p}\right)^{a_m}$$
 $$= \pm \left(\dfrac{q_1}{p}\right)\left(\dfrac{q_2}{p}\right)\ldots \left(\dfrac{q_m}{p}\right).$$
 We have two possibilities:\\
 \\
 \textbf{Case (i):} $2\nmid D$\\
 Then, either $p\equiv 1 \,\bmod 4$ or $p\equiv 3 \,\bmod 4$. If $p\equiv 1\, \bmod 4$ then by quadratic reciprocity, $\left(\dfrac{q_i}{p}\right) = \left(\dfrac{p}{q_i}\right)$. Also, $\left(\dfrac{\pm 1}{p}\right) =1.$  If $p \equiv 3 \bmod 4$, $\left(\dfrac{\pm 1}{p}\right) =-1$ and $\left(\dfrac{q_i}{p}\right) = \pm \left(\dfrac{p}{q_i}\right)$, depending on whether $q_i \equiv 1$ or $ 3\bmod 4$. In general, we can write
 $$ \left(\dfrac{D}{p}\right) = \pm \left(\dfrac{p}{q_1}\right)\left(\dfrac{p}{q_2}\right)\ldots \left(\dfrac{p}{q_m}\right).$$
 Since $p\nmid q$, we know that $p$ is a unit mod $q$, so it is congruent to one of the $q-1$ units in $\Z/q\Z$. We also know that if $q$ is an odd prime, then there are $\dfrac{q-1}{2}$ squares in $(\Z/q\Z)^{\times}$, therefore we conclude that for each $q_i$, the equations $$\left(\dfrac{p}{q_i}\right)=1$$ and $$\left(\dfrac{p}{q_i}\right)=-1$$ each have $\dfrac{q_i-1}{2}$ solutions for $p\, \bmod q_i$.\\
 Let $S^+_i$ denote the set of $\dfrac{q_i-1}{2}$ congruences $\bmod \, q_i$ that solve $\left(\dfrac{p}{q_i}\right)=1$ and $S^-_i$ denote the set of $\dfrac{q_i-1}{2}$ congruences $\bmod \, q_i$ that solve $\left(\dfrac{p}{q_i}\right)=-1$.
\\
 Clearly, 
 \begin{equation}\label{product}
 \left(\dfrac{D}{p}\right) =1 \Leftrightarrow  \pm \left(\dfrac{p}{q_1}\right)\left(\dfrac{p}{q_2}\right)\ldots \left(\dfrac{p}{q_m}\right)=1,
 \end{equation}
  
 Now, the equations $$x_1x_2\ldots x_m =1 \text{ and }x_1x_2\ldots x_m =-1 $$ each have $M=2^{m-1}$ solutions in $\{-1,1\}^m$. 

 Let us enumerate them as
  $$X_1 = (x_{11},x_{12},\ldots,x_{1m})\qquad \qquad\qquad Y_1 = (y_{11},y_{12},\ldots,y_{1m})$$
 $$X_2 = (x_{21},x_{22},\ldots,x_{2m})\qquad \qquad\qquad Y_2 = (y_{21},y_{22},\ldots,y_{2m})$$
 $$\vdots \qquad\qquad\qquad \text{and             }\qquad\qquad \qquad\vdots$$
 $$X_M = (x_{M1},x_{M2},\ldots,x_{Mm}) \qquad\qquad Y_M = (y_{M1},y_{M2},\ldots,y_{Mm})$$
 respectively, where each of the $x_{ij},y_{ij}$ are $1$ or $-1$.
 Depending on whether we need the product in Equation (\ref{product}) to be $1$ or $-1$, we solve using $X_i$'s or $Y_j$'s.
 
 Without loss of generality let us assume that we need the product to be $1$ and that we are in the case $p \equiv 1 \bmod 4.$\\
 Then, for each solution $X_j$, $j=1,\ldots, M$  we need to solve the following system :\\
 $$p\equiv 1\, \bmod 4$$
 $$\left(\dfrac{p}{q_i}\right) = x_{ji},\qquad i=1, \ldots m.$$
 For each $i$, the equation $\left(\dfrac{p}{q_i}\right) = x_{ji}$ will involve choosing a congruence relation among $S^{\pm}_i$ depending on the parity of $x_{ji}$. This gives us a total of $\prod\limits_{i=1}^{m} \dfrac{q_i-1}{2}$ systems of congruences for each $X_j$. By the Chinese Remainder Theorem, each system will give rise to a unique solution. Thus, the total number of solutions we obtain is $$M \prod\limits_{i=1}^{m}\dfrac{q_i-1}{2} = 2^{m-1}\prod\limits_{i=1}^{m}\dfrac{q_i-1}{2} = \dfrac{1}{2}\prod\limits_{i=1}^{m}(q_i-1).$$
 \\
 Similarly we get $\dfrac{1}{2}\prod\limits_{i=1}^{m}(q_i-1)$ solutions coming from the parallel case of $p\equiv 3\, \bmod 4.$\\
 So, in total we have $\prod\limits_{i=1}^{m}(q_i-1)$ number of solutions $(\bmod \, 4q_1q_2\ldots q_m).$ \\
 If we denote $Q= 4q_1q_2\ldots q_m$, then $\left(\dfrac{D}{p}\right) =1 $ has $\dfrac{1}{2} \phi(Q)$ number of solutions $\bmod \, Q$.
 \\
 \textbf{Case (ii):}$2|D.$
 
 Without loss of generality, we may assume that $q_1 = 2$ and $q_i$ is odd for $i=2,\ldots,k.$\\
 Therefore, we need to find solutions to the equation 
 $$\left(\dfrac{D}{p}\right)=\pm \left(\dfrac{2}{p}\right) \left(\dfrac{q_2}{p}\right)\ldots \left(\dfrac{q_m}{p}\right)$$
 $$= \pm \left(\dfrac{2}{p}\right)\left(\dfrac{p}{q_2}\right)\ldots \left(\dfrac{p}{q_m}\right).$$
 The only difference in this case is that instead of considering the congruence $p\equiv 1$ or $3\, \bmod 4$, we further consider congruences $\bmod \, 8$:\\
 If $p\equiv 1 \,\bmod 4$, we have 
 $$\left(\dfrac{2}{p}\right) = 
 \begin{cases}
 1 & \text{if  } p\equiv 1\, \bmod 8\\
 -1 & \text{if  } p\equiv 5\, \bmod 8. 
 \end{cases}
 $$
 Thus in this case, for each $i=2,\ldots m$, we have $\dfrac{q_i-1}{2}$ number of congruences $\bmod \, q_i$ and one congruence $\bmod \,8$ corresponding to $i=1$. Therefore, for every $X_j$ (or $Y_j$, depending on whether we need the product to be $1$ or $-1$) we get $\prod\limits_{i=2}^{m} \dfrac{q_i-1}{2}$ number of solutions. Hence the total number of solutions is $$\prod\limits_{i=2}^{m} {(q_i-1)}.$$
 Similarly, if $p\equiv 3\, \bmod 4$, then we use 
 $$\left(\dfrac{2}{p}\right) = 
 \begin{cases}
 1 & \text{if  } p\equiv  7\, \bmod 8\\
 -1 & \text{if  } p\equiv 3\, \bmod 8 
 \end{cases}
 $$
 and obtain another set of $\prod\limits_{i=2}^{m} (q_i-1)$ solutions.\\
 So we have a total of $$2\prod\limits_{i=2}^{m} (q_i-1) = \dfrac{1}{2} \phi(4q_1q_2\ldots q_m)=\dfrac{1}{2}\phi(Q)$$ solutions, which is the same number as in Case 1.\\
 To summarize, for a fixed non-zero integer $D$, the number of odd primes $p \,\bmod Q$ so that $\left(\dfrac{D}{p}\right)=1$ is $\dfrac{1}{2}\phi(Q)$.
 Coming back to our problem, we wish to calculate $$\#\left\{ \text{primes }p\leq x :  \left( \frac{D}{p} \right)=1 \right\}.$$
 By the natural density statement of Dirichlet's theorem, we know that for any positive integer $a$ which is coprime to $n$,
  $$\# \{p\leq x,p \text{ prime } | p\equiv a \,\bmod n\} \sim \dfrac{1}{\phi(n)}\pi(x).$$
 Let $B(1) := \{b_i, i=1,\ldots,b_{\frac{\phi(Q)}{2}}\}$ denote the set solutions $\bmod \, Q$ obtained from the discussion above and $B(-1):= \{b'_i, i=1,\ldots,b'_{\frac{\phi(Q)}{2}}\}$ denote the remaining residue classes that correspond to the primes $p \bmod Q$ so that $\left(\dfrac{D}{p}\right)=-1$. Then, $\left( \frac{D}{p} \right)=1$ if and only if $p$ is congruent to any one of the elements in the set $B(1)$. So we have\\
 $$\#\left\{ \text{primes }p\leq x :  \left( \frac{D}{p} \right)=1 \right\} = \sum\limits_{i=1}^{\frac{\phi(Q)}{2}} \# \{p\leq x,p \text{ prime } | p\equiv b_i \,\bmod Q\} \sim \sum\limits_{i=1}^{\frac{\phi(Q)}{2}}\dfrac{1}{\phi(Q)}\pi(x)= \dfrac{1}{2} \pi(x).$$
 Hence, the asymptotic density of primes $p$ for which $\left( \frac{D}{p} \right)=1$ is $\dfrac{1}{2}.$
 \\
 
Using the set $B(-1)$, the same proof can be used to show that 
$$\#\left\{ \text{primes }p\leq x :  \left( \frac{D}{p} \right)=-1 \right\} \sim \dfrac{1}{2} \pi(x),$$
implying that the density of primes $p$ for which $f(x)$ has no solution $\bmod \,p$ is $\dfrac{1}{2}$.

 %Extending this result to the case of squarefree numbers with exactly $k$ factors, we have the following result.\\
%\begin{lemma} \label{sqfree-k}
%For a fixed $D \in \N$, 
%$$\# \bigg\{ \text{Squarefree }n\leq x, n=p_1p_2\ldots p_k: \left(\dfrac{D}{p_i}\right)=1 \text{ for each } i  \bigg\} \sim \dfrac{\pi_k(x)}{2^k} \sim \dfrac{x (\log\log x)^{k-1}}{2^k (k-1)!\log x}.$$
%\end{lemma}
\noindent
We now use this proposition to prove Theorem \ref{main-2}.\\
\\
\textbf{Remark:} 
 From the statement of Proposition \ref{Prime} and Theorem \ref{main-2}, it is clear that we are counting only those squarefree numbers with $k$-prime factors which are coprime to the discriminant $D$ of $f(x)$. \\
 \\
\textbf{Proof of Theorem \ref{main-2}.}
We first prove the statement for $n$ odd.\\In this case, using Proposition \ref{Prime} we conclude that the condition $$ \left(\dfrac{D}{p_i}\right)= \varepsilon_i \text{ for each } i  $$ will hold if and only if every prime $p_i$ dividing $n$ belongs to the set $B(\varepsilon_i)$.

Let us represent the (odd) squarefree number as a tuple $(p_1,p_2,\ldots,p_k)$ with $p_1<p_2<\ldots<p_k$ and choose any $k$-tuple $(m_1,m_2,\ldots,m_k)$ where each $m_i \in B(\varepsilon_i)$. Since $|B(\pm 1)| = \dfrac{\phi(Q)}{2}$, the number of $k$-tuples such that 
\begin{equation} \label{condition}
(p_1,p_2,\ldots,p_k) \equiv (m_1,m_2,\ldots,m_k)\, \bmod Q
\end{equation}
 component-wise is $\left(\dfrac{\phi(Q)}{2}\right)^k$. Therefore, appplying Theorem \ref{mainthm}, we have
$$ \# \bigg\{\text{Odd }n\leq x, n=p_1p_2\ldots p_k \text{ with } p_1<p_2<\ldots<p_k: \left(\dfrac{D}{p_i}\right)=\varepsilon_i \text{ for each } i  \bigg\} \sim \dfrac{1}{\phi(Q)^k} \dfrac{x (\log\log x)^{k-1}}{(k-1)!\log x}\left(\dfrac{\phi(Q)}{2}\right)^k,$$ settling the odd case.\\
\textbf{Note:} Even $n$ are counted only if $D$ is odd.\\
The even case follows by counting the number of odd squarefree $n \leq x/2 $ with $k-1$ prime factors. From the argument for the odd case, we have \\
$$ \# \bigg\{ n\leq x, n=2p_2\ldots p_k, \text{ with } 2=p_1<p_2<\ldots<p_k: \left(\dfrac{D}{p_i}\right)=\varepsilon_i \text{ for each } i  \bigg\} \sim \dfrac{1}{\phi(Q)^{k-1}} \pi_{k-1}(x/2)\left(\dfrac{\phi(Q)}{2}\right)^{k-1}.$$
Noting that $\dfrac{\pi_{k-1}(x/2)}{2^{k-1}} = o\left(\dfrac{\pi_{k}(x)}{2^k}\right),$ the result follows.\\

\begin{cor}
The density of squarefree numbers $n$ with k prime factors so that a quadratic equation has exactly $2^k$ solutions mod n is $\dfrac{1}{2^k}$.
\end{cor}
\begin{proof}
This easily follows from Theorem \ref{main-2} by taking $D$ as the discriminant of the quadratic equation and $\underline{\varepsilon}$ with $\varepsilon_i =1$ for each $i$.
\end{proof}
\noindent
\textbf{Note:}
We may ask what happens when $n$ has $k$ prime factors counted with multiplicity, i.e., when $n = p_1p_2\ldots p_k$ is not necessarily squarefree. We observe that in this case, the $k$-tuple $(m_1,m_2,\ldots,m_k)$ will neccesarily have $m_i = m_j$ whenever $p_i=p_j$. Therefore, for such $n$, the number of $k$-tuples satisfying Equation \ref{condition} will be bounded by $ \left(\dfrac{\phi(Q)}{2}\right)^{k}$ and equal to it if and only if $n$ is squarefree. Hence, we deduce the following:
\begin{cor} Let $D \in \Z-\{0\}$ and $k \in \N$. For any $k$-tuple $\underline{\varepsilon} = (\varepsilon_1, \ldots, \varepsilon_k)$ where each $\varepsilon_i = \pm 1$ for each $i=1,\ldots,k$, we have
 $$\# \bigg\{ n\leq x: n=p_1p_2\ldots p_k \text{ with } p_1\leq p_2\leq \ldots \leq p_k: \left(\dfrac{D}{p_i}\right)=\varepsilon_i \text{ for each prime } p_i|n  \bigg\} = O \left( \dfrac{1}{2^k}\tau_k(x)\right),$$ where $\tau_k(x)$ is the function defined in the introduction.
\end{cor}
\bigskip
\noindent
%\textbf{Acknowledgements}

%I thank my advisor Dr. Kaneenika Sinha for encouraging me to write this article and giving her comments and Prof. Ram Murty for his comments and suggesting some improvements. I also thank the referee for useful remarks.


\begin{thebibliography}{999}

\bibitem{HW}G. H. Hardy and E. M. Wright: {An Introduction to the Theory of Numbers}, Oxford
University Press, 1985. Zbl 1159.11001, MR 2445243.

\bibitem{Kl} H. Kornblum, and E. Landau:{ \"{U}ber die Primfunktionen in einer Arithmetischen
Progression}, Math. Zeitschr. {5} (1919), 100--111. JFM 47.0154.02, MR 1544375.  

\bibitem{Landau}E. Landau: {Sur quelques probl\`{e}mes relatifs \`{a} la distribution des nombres premiers}, S.M.F. Bull., {28} (1900) 25--38. JFM 31.0020.01, MR 1504359.

\bibitem{MV} H. L. Montgomery, R. C. Vaughan: {Mutliplicative Number Theory I. Classical Theory}, Cambridge Studies in Advanced Mathematics 97, Cambridge University Press, 2007. Zbl 1245.11002, MR 2378655.

\bibitem{CPom}C. Pomerance: {On the distribution of amicable numbers}, J. reine
angew. Math. {293/294 }(1977), 217--222. Zbl 0349.10004, MR 0447087.

%\bibitem{Sathe} L. G. Sathe, \textit{On a problem of Hardy on the distribution of integers with a given number of prime factors}, J. Indian Math. Soc. \textbf{17} (1953), 63-141 \textbf{18} (1954), 27-81.

% \bibitem{Selberg} A. Selberg, \textit{Note on a paper of L. G. Sathe}, J. Indian Math. Soc. \textbf{18} (1954), 83-87.
\bibitem{Rib} P. Ribenboim: {The New Book of Prime Number Records}, Springer-Verlag, New York, 1996. Zbl 0856.11001, MR1377060
\bibitem{Wright}  E. M. Wright:{ A simple proof of a theorem of Landau}, Proc. Edinburgh Math. Soc. (2) (1954), 9:87--90. Zbl 0057.28601, MR 0065579.

\end{thebibliography}
\end{document}